\documentclass[reqno,11pt,a4paper]{amsart}

\usepackage{amsmath, amsthm, amsfonts, amssymb, color}

\usepackage{textpos}

\usepackage{mathtools}
\mathtoolsset{showonlyrefs}

\usepackage[applemac]{inputenc}

\theoremstyle{plain}
\newtheorem{theorem}{Theorem}
\newtheorem{proposition}[theorem]{Proposition}
\newtheorem{lemma}[theorem]{Lemma}
\newtheorem{corollary}[theorem]{Corollary}

\theoremstyle{definition}
\newtheorem{example}[theorem]{Example}

\theoremstyle{remark}
\newtheorem{remark}[theorem]{Remark}

\def\R{\mathbb{R}}	
\renewcommand{\leq}{\leqslant} 		
\renewcommand{\geq}{\geqslant}

\def\cA{\mathcal{A}}
\def\cB{\mathcal{B}}

\def\hcA{\hat{\cA}}
\def\cF{\mathcal{F}}
\def\hcF{\hat{\cF}}

\def\cV{\mathcal{V}}
\DeclareMathOperator{\der}{Der}
\def\ad{\mathrm{ad}}

\def\fp{\mathfrak{p}}
\def\cW{\mathcal{W}}

\begin{document}

\title[Normal forms of Poisson brackets]{Normal forms of dispersive scalar Poisson brackets with two independent variables}
\author{Guido Carlet}
\address{Korteweg-de Vries Instituut voor Wiskunde, 
Universiteit van Amsterdam, Postbus 94248,
1090GE Amsterdam, Nederland}
\email{g.carlet@uva.nl}

\author{Matteo Casati}
\address{Marie Curie fellow of the Istituto Nazionale di Alta Matematica, Department of Mathematical Sciences,
Loughborough University,
Loughborough LE11 3TU, United Kingdom}
\email{M.Casati@lboro.ac.uk}

\author{Sergey Shadrin}
\address{Korteweg-de Vries Instituut voor Wiskunde, 
Universiteit van Amsterdam, Postbus 94248,
1090GE Amsterdam, Nederland}
\email{s.shadrin@uva.nl}

\date{}
\begin{abstract}
We classify the dispersive Poisson brackets with one dependent variable and two independent variables, with leading order of hydrodynamic type, up to Miura transformations. We show that, in contrast to the case of a single independent variable for which a well known triviality result exists, the Miura equivalence classes are parametrised by an infinite number of constants, which we call numerical invariants of the brackets. We obtain explicit formulas for the first few numerical invariants. 
\end{abstract}
\maketitle

%
%\begin{textblock*}{8cm}(10cm,-7.8cm)
 %  \fbox{\footnotesize \sc \jobname.tex}
%\end{textblock*}
%

%\tableofcontents

\section*{Introduction}

Let $\cA$ be the space of differential polynomials in the variable $u$, i.e., formal power series in the variables $\partial_{x^1}^{k_1}\partial_{x^2}^{k_2} u$ with coefficients which are smooth functions of $u$:
\begin{equation}
\cA = C^\infty (U) [[ \{ u^{(k_1,k_2)} =\partial_{x^1}^{k_1} \partial_{x^2}^{k_2} u \text{ with } k_1, k_2 \geq0 , \,(k_1,k_2 ) \not= (0,0) \} ]],
\end{equation} 
for $U \subset\R$.
The standard degree $\deg$ on $\cA$ counts the number of derivatives $\partial_{x^1}$, $\partial_{x^2}$ in a monomial, i.e., it is defined by $\deg ( \partial_{x^1}^{k_1} \partial_{x^2}^{k_2} u) = k_1 +  k_2$.

In this paper we classify, up to Miura transformations, the dispersive Poisson brackets with one dependent variable $u$ and two independent variables $x^1$, $x^2$ of the form
\begin{align} \label{def}
\{ &u(x^1,x^2) ,u(y^1,y^2) \} = \{ u(x^1,x^2) , u(y^1,y^2) \}^0 +\\ &+\sum_{k>0} \epsilon^k \sum_{\substack{k_1,k_2 \geq0 \\ k_1 +k_2 \leq k+1}}  A_{k; k_1, k_2}(u(x)) \delta^{(k_1)}(x^1-y^1)  \delta^{(k_2)}(x^2-y^2)  \notag
\end{align}
where $A_{k; k_1,k_2} \in \cA$ and $\deg A_{k; k_1,k_2} = k - k_1 - k_2 +1$.

The leading term $ \{ u(x^1,x^2) , u(y^1,y^2) \}^0$ is a (scalar, two-dimensional) Poisson bracket of Dubrovin-Novikov (or hydrodynamic) type~\cite{dn83, dn84}, in other words it is of the form
\begin{equation}
 \{ u(x^1,x^2) , u(y^1,y^2) \}^0 = \sum_{i=1}^2 
\left[ g^i(u(x)) \partial_{x^i} +
b^i(u(x)) \partial_{x^i} u(x) \right] 
\delta(x^1-y^1)  \delta(x^2-y^2) ,
\end{equation}
which we assume to be non-degenerate.

The conditions imposed on the functions $g^i(u)$ and $b^i(u)$ by the requirement that $\{ , \}^0$ is skew-symmetric and satisfies the  Jacobi identity have been studied by several authors~\cite{m88, m08, fls15}. We require the additional condition that the bracket is non-degenerate, namely that the bracket does not vanish for any value of the function $u(x)$. 
In the specific case considered here, where there is a single dependent variable and two independent variables, such conditions guarantee the existence of a change of coordinates in the dependent variable (a Miura transformation of the first kind), to a flat coordinate that we still denote with $u$, in which the bracket assumes the form
\begin{align}
\{ u(x^1,x^2) , u(y^1,y^2) \}^0 &= c^1 \delta^{(1)}(x^1-y^1)  \delta(x^2-y^2) + \\ \notag
&+ c^2 \delta(x^1-y^1)  \delta^{(1)}(x^2-y^2) .
\end{align}
We can moreover perform (see~\cite{ccs15}) a linear change in the independent variables $x^1$, $x^2$ such that the Poisson bracket assumes the standard form
\begin{equation} \label{leading}
 \{ u(x^1,x^2) , u(y^1,y^2) \}^0 = \delta (x^1-y^1)  \delta^{(1)}(x^2-y^2) .
\end{equation}

The Miura transformations (of the second kind~\cite{lz11}) are changes of variable of the form
\begin{equation} \label{miura}
v = u + \sum_{k\geq1} \epsilon^k F_k
\end{equation}
where $F_k \in \cA$ and $\deg F_k= k$. They form a group called Miura group. We say that two Poisson brackets which are mapped to each other by a Miura transformation are Miura equivalent.

As follows from the discussion so far, the classification of dispersive Poisson brackets of the form~\eqref{def} (with non-degeneracy condition) under Miura transformations~\eqref{miura}, diffeomorphisms of the dependent variable and linear changes of the independent variables reduces to the problem of finding the normal forms of the equivalence classes under Miura transformations of the second kind~\eqref{miura} of the Poisson brackets~\eqref{def} with leading term~\eqref{leading}.

We solve this problem in our main result:
\begin{theorem} \label{maintheorem}
The normal form of Poisson brackets~\eqref{def} with leading term~\eqref{leading} under Miura transformations of the second kind  is given by
\begin{align} \label{normalform}
\{ u(x^1,x^2) ,u(y^1,y^2) \} &=  \delta (x^1-y^1)  \delta^{(1)}(x^2-y^2)+\notag \\
&+ \sum_{k \geq 1} \epsilon^{2k+1} c_k \delta^{(2k+1)}(x^1-y^1)  \delta(x^2-y^2)
\end{align}
for a sequence of constants  $c=(c_1, c_2, \dots )$.
\end{theorem}

\begin{remark} \label{steps}
By ``normal form'', in the main Theorem, we mean that: 
\begin{enumerate}
\item[i.]
for any choice of constants $c_k$ formula~\eqref{normalform} defines a Poisson bracket which is a deformation of~\eqref{leading}; 
\item[ii.]
two Poisson brackets of the form~\eqref{normalform} are Miura equivalent if and only if they are defined by the same constants $c_k$; 
\item[iii.]
and any Poisson bracket of the form~\eqref{def} can be brought to the normal form~\eqref{normalform} by a Miura transformation. 
\end{enumerate}
We call the constants $c_k$ the {\it numerical invariants} of the Poisson bracket.
\end{remark}

The deformation theory of Hamiltonian -- and, albeit not addressed in our paper, bi-Hamiltonian -- structures plays an important role in the classification of integrable Hamiltonian PDEs~\cite{dz01,dlz06}. Most results in this field have been obtained for $(1+1)$-dimensional systems, namely the ones that depend only on one space variable. 

The main result in this line of research is the triviality theorem~\cite{g02, dms02, dz01} of Poisson brackets of Dubrovin-Novikov type.
%The main result of that line of research is due to E.~Getzler~\cite{g02}, who proved the triviality of the Poisson-Lichnerowicz cohomology associated to the brackets of Dubrovin-Novikov type; independent results by Degiovanni et al.~\cite{dms02} focus on the problem of the deformations and obtain only the first and second cohomology groups.
Together with the classical results by Dubrovin and Novikov~\cite{dn83}, this allows to conclude that the  dispersive deformations of non-degenerate Dubrovin-Novikov brackets are classified by the signature of a pseudo-Riemannian metric.
Similarly, deformations of bi-Hamiltonian pencils~\cite{l02,b08} are parametrised by functions of one variable, the so-called central invariants~\cite{dlz06, dlz08}; in a few special cases, the corresponding biHamiltonian cohomology has been computed, in particular for scalar brackets~\cite{lz13, cps16, cps16-2}, and in the semi-simple $n$-component case~\cite{cps17, cks17} . The $(2+1)$-dimensional case is much less studied: the classification of the structures of hydrodynamic type has been completed up to the four-components case~\cite{fls15}, while in our recent paper~\cite{ccs15} we computed the Poisson cohomology for scalar -- namely, one-component -- brackets. Since such a cohomology is far from being trivial, the actual classification of the dispersive deformations of such brackets is a highly complicated task. We address and solve it in the present paper.

The outline of the paper is as follows: 
in Section 1 we quickly recall basic definitions and facts related with the theta formalism. 
In Section 2 we specialise some results from our previous work~\cite{ccs15} to the $D=2$ case to obtain an explicit description of the second Poisson cohomology. 
In Section 3 we prove our main result. The proof is split in three steps corresponding to the three parts in Remark~\ref{steps}. In \S\ref{techsec} we prove some technical lemmas that are required in the proof of Proposition~\ref{step3}. 
Finally in Section 4 we give an explicit expression of the first few numerical invariants of the Poisson bracket. 

\subsubsection*{Acknowledgment} 

We would like to thank Jenya Ferapontov for several useful observations and Dario Merzi for suggesting a clever identity in Example~\ref{eg}. 
G.~C. and S.~S. were supported by Netherlands Organisation for Scientific Research (NWO). 
M.~C. was supported by the INdAM-COFUND-2012 Marie Curie fellowship ``MPoisCoho -- Poisson cohomology of multidimensional Hamiltonian operators''.

\section{Theta  formalism}

We present here a short summary of the basic definitions of the theta formalism for local variational multivector fields, specialising the formulas to the scalar case with two independent variables, i.e., $N=1$, $D=2$. We refer the reader to~\cite{ccs15} for the general $N, D$ case. 

Let $\cA$ be the space of differential polynomials
\begin{equation}
\cA = C^\infty (\R) [[ \{ u^{(s,t)} , s, t \geq0 , \ (s, t) \not= (0,0) \} ]],
\end{equation}
where we denote $u^{(s,t)} = \partial_x^s \partial_y^t u$, and $C^\infty (\R)$ denotes the space of smooth functions in the variable $u$. 
The standard gradation $\deg$ on $\cA$ is given by $\deg u^{(s,t)} = s+t$. We denote $\cA_d$ the homogeneous component of degree $d$.
%
%On $\cA$ we define commuting derivations $\partial_{x}$, $\partial_y$  by 
%\begin{equation}
%\partial_{x^i} = \sum_{\alpha, S} u^{\alpha, S+\xi_i} \frac{\partial }{\partial u^{\alpha, S} }.
%\end{equation}

Using the standard derivations $\partial_x$ and $\partial_y$ on $\cA$, we define the space of local functionals as
\begin{equation}
\cF = \frac{\cA}{\partial_{x} \cA  + \partial_{y}\cA} ,
\end{equation}
and the projection map from $\cA$ to $\cF$ is denoted by a double integral, which associates to $f\in\cA$ the element
\begin{equation}
\int f \ dx \ dy
\end{equation}
in $\cF$. Moreover, we will denote by the partial integrals $\int dx$ , $\int dy$ the projections from $\cA$ to the quotient spaces $\cA/\partial_{x}\cA$, $\cA/\partial_{y}\cA$.

The variational derivative of  a local functional $F=\int f $ is defined as 
\begin{equation} \label{varu}
\frac{\delta F}{\delta u} = \sum_{s,t \geq0} (- \partial_x)^s (- \partial_y)^t\frac{\partial f}{\partial u^{(s,t)}} .
\end{equation}

A local $p$-vector $P$ is a linear $p$-alternating map from $\cF$ to itself of the form
\begin{equation} \label{pvect}
P(I_1, \dots ,I_p) = \int 
P_{(s_1,t_1), \dots , (s_p,t_p)}
 \ \partial_x^{s_1} \partial_y^{t_1}  \left( \frac{\delta I_1}{\delta u} \right) \cdots  \partial_x^{s_p} \partial_y^{t_p} \left( \frac{\delta I_p}{\delta u} \right)
 \ dx \ dy
\end{equation}
where $P_{(s_1,t_1), \dots , (s_p,t_p)}  \in \cA$, for arbitrary $I_1, \dots, I_p \in \cF$. We denote the space of local $p$-vectors by $\Lambda^p \subset \mathrm{Alt}^p(\cF, \cF)$.

Clearly an expression of the form~\eqref{def} defines a local bivector by the usual formula 
\begin{equation}
\{ I_1, I_2 \} = \int \frac{\delta I_1}{\delta u(x^1, y^1)} 
\{ u(x^1, y^1), u(x^2,y^2) \} 
\frac{\delta I_2}{\delta u(x^2,y^2)} 
\ dx^1 \ dy^1 \ dx^2 \ dy^2  
\end{equation}
which equals to 
\begin{equation}
\int \sum_{k\geq0} \epsilon^k \frac{\delta I_1}{\delta u(x,y)}
\sum_{\substack{s,t \geq0 \\ s +t \leq k+1}}  
A_{k; s, t}(u(x)) \partial_x^s \partial_y^t  \frac{\delta I_2}{\delta u(x,y)} \ dx \ dy.
\end{equation}

The theta formalism, introduced first in the context of formal calculus of variations in~\cite{g02}, can be easily extended to the multi-dimensional setting~\cite{ccs15}, and allows to treat the local multivectors in a more algebraic fashion.

We introduce the algebra $\hcA$ of formal power series in the commutative variables $u^{(s,t)}$ and anticommuting variables $\theta^{(s,t)}$, with coefficients given by smooth functions of $u$, i.e.,
%
%To use the $\theta$-formalism, and to avoid working with the densities of \eqref{pvect}, we introduce a new space. Let $\hcA$ be the algebra of formal power series in the commutative variables $u^{\alpha,S}$, $|S|>0$ and anticommutative variables 
%$\theta_\alpha^S$, $|S|\geq0$ with coefficients given by smooth functions of the variables $u^\alpha$, i.e.,
\begin{equation}
\hcA := C^\infty (\R) [[ \left\{ u^{(s,t)},(s,t)\not=(0,0) \right\}\cup \left\{ \theta^{(s,t)} \right\}]] .
\end{equation}
The standard gradation $\deg$ and the super gradation $\deg_\theta$ of $\hcA$ are defined by setting
\begin{equation}
\deg u^{(s,t)} =\deg \theta^{(s,t)} = s+t, \quad 
\deg_\theta u^{(s,t)} = 0 , \quad 
\deg_\theta \theta^{(s,t)} = 1 .
\end{equation}
We denote $\hcA_d$, resp. $\hcA^p$, the homogeneous components of standard degree $d$, resp. super degree $p$, while $\hcA_d^p := \hcA_d \cap \hcA^p$. Clearly $\hcA^0 =\cA$.
The derivations $\partial_x$ and $\partial_y$ are extended to $\hcA$ in the obvious way.

%The commuting derivations $\partial_{x^i}$ for $i=1, \dots , D$ are extended to $\hcA$ by 
%\begin{equation}
%\partial_{x^i} = \sum_{\alpha, S} \left( u^{\alpha, S+\xi_i} \frac{\partial }{\partial u^{\alpha, S} } + \theta_\alpha^{S+\xi_i} \frac{\partial }{\partial \theta_\alpha^S} \right).
%\end{equation}

We denote by $\hcF$ the quotient of $\hcA$ by the subspace $\partial_{x} \hcA + \partial_{y} \hcA$, and by a double integral $\int dx \ dy$ the projection map from $\hcA$ to $\hcF$. Since the derivations $\partial_x$, $\partial_y$ are homogeneous, $\hcF$ inherits both gradations of $\hcA$.  

It turns out, see Proposition 2 in~\cite{ccs15}, that the space of local multivectors $\Lambda^p$ is isomorphic to $\hcF^p$ for $p\not=1$, while $\Lambda^1$ is isomorphic to the quotient of $\hcF^1$ by the subspace of elements of the form $ \int (k_1 u^{(1,0)} + k_2 u^{(0,1)} ) \theta $ for two constants $k_1$, $k_2$. 
Moreover $\hcF^1$ is isomorphic to the space $\der'(\cA)$ of derivations of $\cA$ that commute with $\partial_x$ and $\partial_y$.

%\begin{proposition}[\cite{ccs15}, Proposition 2]
%The space of local multi-vectors $\Lambda^p$ is isomorphic to $\hcF^p$ for $p\not=1$. Moreover
%\begin{equation}
%\Lambda^1 
%\cong \frac{\hcF^1}{\oplus_i \R \int u^{\alpha, \xi_i} \theta_\alpha} 
%\cong \frac{\der'(\cA)}{\oplus_i \R \partial_{x^i}},
%\end{equation}
%where $\der'(\cA)$ denotes the space of derivations of $\cA$ that commute with $\partial_{x^i}$, for $i=1,\dots ,D$, and $\der'(\cA) \cong \hcF^1$.
%\end{proposition}

%\begin{remark}
%First of all, for $p=0$ the isomorphism is trivial, since $\hcF^0=\cF=\Lambda^0$. In general, this proposition allows us to replace the standard formalism for local multivector fields, involving Dirac's $\delta$ functions, with the so-called $\theta$-formalism introduced by Getzler \cite{g02} and extended to the multi-dimensional setting in our previous work \cite{ccs15}.
%\end{remark}

%\subsection{The Schouten-Nijenhuis bracket}

The Schouten-Nijenhuis bracket 
\begin{equation}
[\phantom{A},\phantom{A}]:\hcF^p \times \hcF^q \to \hcF^{p+q-1}
\end{equation}
is defined as  
\begin{equation}\label{defsch}
[P , Q ] = \int^D \left( \frac{\delta P}{\delta\theta} \frac{\delta Q}{\delta u} + (-1)^p \frac{\delta P}{\delta u} \frac{\delta Q}{\delta \theta} \right) \ dx \ dy ,
\end{equation}
where the variational derivative with respect to $\theta$ is defined as
\begin{equation}
\frac{\delta}{\delta \theta} = \sum_{s,t\geq0} (-\partial_x)^s (-\partial_y)^t \frac{\delta}{\delta \theta^{(s,t)}} .
\end{equation}

It is a bilinear map that satisfies the graded symmetry 
\begin{equation}
[P,Q] = (-1)^{pq} [Q,P]
\end{equation}
and the graded Jacobi identity
\begin{equation}
(-1)^{pr} [[P,Q],R] + (-1)^{qp} [[Q,R],P] + (-1)^{rq} [[R,P],Q] =0 
\end{equation}
for arbitrary $P\in\hcF^p$, $Q\in\hcF^q$ and $r\in\hcF^r$.

%\begin{remark}
%From the properties of the variational derivative $\delta_u\partial=\delta_\theta\partial=0$, we can define the Schouten-Nijenhuis bracket as an operation $\hcA^p\times\hcA^q\to\hcF^{p+q-1}$.
%\end{remark}
A bivector $P\in\hcF^2$ is a Poisson structure when $[P,P]=0$. In such case $d_P := ad_P = [P, \cdot ]$ squares to zero, as a consequence of the graded Jacobi identity, and the cohomology of the complex $(\hcF, d_P)$ is called Poisson cohomology of $P$. 

The Miura transformations of the second kind~\cite{lz11} are changes of variable of the form
\begin{equation}\label{eq:Miuradef}
 u \mapsto \tilde{u} =\sum_{k=0}^\infty\epsilon^k F_k(u)
\end{equation}
on the space $\cA$, where $F_k\in\cA_k$. They form a subgroup of the general Miura group~\cite{dz01} which also contains the diffeomorphisms of the variable $u$. The action of a general Miura transformation of the second kind on a local multivector $Q$ in $\hcF$ is given by the exponential of the adjoint action with respect to the Schouten-Nijenhuis bracket
\begin{equation*}
e^{\ad_X}Q=Q+[X,Q]+\frac{1}{2}[X,[X,Q]]+\frac16[X,[X,[X,Q]]]+\cdots,
\end{equation*}
where $X\in\hcF^1_{\geq1}$ is a local vector field such that $e^{\ad_X} u = \tilde{u}$.

\section{Poisson cohomology}
\label{sec:poisson}

In our previous paper~\cite{ccs15} we gave a description of the Poisson cohomology of a scalar multidimensional Poisson bracket in terms of the cohomology of an auxiliary complex with constant coefficients. 
Our aim here is to give an explicit description of a set of generators of the Poisson cohomology in the $D=2$ case, which will be used in the proof of the main theorem in the next Section. 

Let us begin by recalling without proof a few results from our paper~\cite{ccs15}, specialising them to the case $D=2$.

Consider the short exact sequences of differential complexes
\begin{equation} \label{firstses}
0 \to \hcA / \R \xrightarrow{\partial_x}  \hcA \xrightarrow{\int dx}  \hcF_1 \to 0,
\end{equation}
\begin{equation} \label{secondses}
0 \to \hcF_1 / \R \xrightarrow{\partial_y}  \hcF_1 \xrightarrow{\int dy}  \hcF \to 0 ,
\end{equation}
where the differential is induced an all spaces by 
\begin{equation}
\Delta = \sum_{s,t \geq0} \theta^{(s,t+1)} \frac{\partial }{\partial u^{(s,t)}}. 
\end{equation}
On $\hcF$ such differential coincides with $\ad_{\fp_1}$, where $\fp_1=\frac{1}{2}\int\theta\theta^{(0,1)}dxdy$.

In the long exact sequence in cohomology associated with~\eqref{firstses} the Bockstein homomorphism vanishes, therefore
\begin{equation}
H(\hcF_1) = \frac{H(\hcA)}{\partial_x H(\hcA)}.
\end{equation}
Moreover, the cohomology classes in $H(\hcA)$ can be uniquely represented by elements of the polynomial ring $\Theta$ generated by the anticommuting variables $\theta^{(s,0)}$, $s\geq0$ with real coefficients.

The map induced in cohomology by the map $\partial_y$ in the short exact sequence~\eqref{secondses} vanishes, therefore we get the following exact sequence
\begin{equation} \label{sest}
%0 \to H^p_d(\hcF_1)\to H^p_d(\hcF) \to H^{p+1}_d(\hcF_1/\R) \to 0 ,
0 \to \left( \frac{\Theta}{\partial_x \Theta}\right)^p_d 
\xrightarrow{\int dy}  H^p_d(\hcF) 
\to \left( \frac{\Theta}{\partial_x \Theta}\right)^{p+1}_d \to 0 ,
\end{equation} 
where the third arrow is the Bockstein homomorphism. 

This sequence allows us to write the Poisson cohomology $H^p(\hcF)$ as a sum of two homogeneous subspaces of $\Theta \slash \partial_x \Theta$ in super-degree $p$ and $p+1$ respectively, where the first one is simply injected, while the second one has to be reconstructed via the inverse to the Bockstein homomorphism. 

Let $\iint a \ dx \ dy \in \hcF^p_d$ be an $\ad_{\fp_1}$-cocycle. Then, there exist $b, b' \in \hcA^{p+1}_d$ such that 
\begin{equation}
\Delta a = \partial_y b + \partial_x b'.
\end{equation}
The Bockstein homomorphism assigns to the cocycle $\iint a \ dx \ dy$ the cocycle $\int b \ dx \in \hcF^{p+1}_d$.

Let us define a map $\cB: \Theta \to \hcA$ by
\begin{equation} \label{BBdef}
\cB = \sum_{i\geq0}u^{(i,0)}\frac{\partial }{\partial \theta^{(i,0)}}, 
\end{equation}
which clearly commutes with $\partial_x$, and therefore induces a map 
from $\frac{\Theta}{\partial_x \Theta}$ to $\hcF$. 
We have that
\begin{equation}
\Delta \cB = \partial_y ,
\end{equation}
and consequently, $\cB$ defines a splitting map 
\begin{equation}
\cB : \left( \frac{\Theta}{\partial_x \Theta}\right)^{p+1}_d \to H^p_d(\hcF) 
\end{equation}
for the short exact sequence~\eqref{sest}.

We have therefore shown that
\begin{lemma}
$H^p_d(\hcF) = \left( \frac{\Theta}{\partial_x \Theta}\right)^{p}_d \oplus \cB \left( \frac{\Theta}{\partial_x \Theta}\right)^{p+1}_d .
$\end{lemma}
We remark that this lemma gives an explicit description of representatives of the cohomology classes in $H^p_d(\hcF)$. 
In particular, the only non-trivial classes in $\Theta \slash \partial_x \Theta$ in super-degree $p=2$ are given by $\theta \theta^{(2k+1,0)}$ for $k\geq1$, and correspond to the deformations of the Poisson brackets in Theorem~\ref{maintheorem}. 
The following reformulation of this observation will be useful in the proof of Proposition~\ref{step3}:
\begin{corollary} \label{corcoh}
\begin{equation}
H^2_{2k}(\hcF) = \cB \left( \frac{\Theta}{\partial_x \Theta}\right)^{3}_{2k}, 
\end{equation}
\begin{equation}
H^2_{2k+1}(\hcF) = \R \theta \theta^{(2k+1,0)} \oplus \cB \left( \frac{\Theta}{\partial_x \Theta}\right)^{3}_{2k+1}.
\end{equation}
\end{corollary}
Moreover, we can define an explicit basis of $\left( \frac{\Theta}{\partial_x \Theta}\right)^{3}_{d}$ and $\cB \left( \frac{\Theta}{\partial_x \Theta}\right)^{3}_{d}$:
\begin{lemma} \label{lembas}
A basis of $\left( \frac{\Theta}{\partial_x \Theta}\right)^{3}_{d}$ is given by representatives
%\begin{align*}
%	& \theta^{k+1}\theta^k\theta^0, \theta^{k}\theta^{k-1}\theta^2, \theta^{k-1}\theta^{k-2}\theta^4,\cdots  & \mbox{for } d=2k+1, \\
%	& \theta^{k}\theta^{k-1}\theta^1, \theta^{k-1}\theta^{k-2}\theta^3, \theta^{k-2}\theta^{k-3}\theta^5,\cdots  & \mbox{for } d=2k,
%	\end{align*}
\begin{align}
&\theta^{k-l} \theta^{k-l-1} \theta^{2l} , \quad l=0,\dots , \lfloor \frac{k-2}3 \rfloor, \quad \text{for $d=2k-1$} , \\
&\theta^{k-l} \theta^{k-l-1} \theta^{2l+1} , \quad l=0,\dots , \lfloor \frac{k-3}3 \rfloor, \quad \text{for $d=2k$},
\end{align}
where we use the notation $\theta^k = \theta^{(k,0)}$.
\end{lemma}
\begin{proof}
More generally we can prove that a basis of $\left( \frac{\Theta}{\partial_x \Theta}\right)^{p}_{d}$ is given by 
\begin{equation} \label{quotbas}
\theta^{i_2+1} \theta^{i_2} \theta^{i_3} \cdots \theta^{i_p}  
\end{equation}
with 
\begin{equation}
i_2 > i_3 >\dots  > i_p \geq 0, \quad 
1+ 2i_2 + i_3 + \dots + i_p = d.  
\end{equation}

A basis of $\Theta^p_d$ is given by monomials $\theta^{i_1} \cdots \theta^{i_p}$ with 
\begin{equation}
i_1 > i_2 >\dots  > i_p \geq 0 , \quad
i_1 + \dots +i_p = d.
\end{equation}
We arrange such monomials in {\it lexicographic order}, that is, we say that $\theta^{i_1} \cdots \theta^{i_p} > \theta^{j_1} \cdots \theta^{j_p}$ if $i_1 > j_1$, or if $i_1 = j_1$ and $i_2 > j_2$, and so on. 

For an element  $a = \theta^{i_1} \cdots \theta^{i_p}$ of the basis of $\Theta^p_{d-1}$, we have that the leading term (in lexicographic order)  of $\partial_x a$ is given by 
\begin{equation} \label{topbas}
(\partial_x a)^{top}=\theta^{i_1+1} \theta^{i_2} \cdots \theta^{i_p} .
\end{equation}
Note that if $a_1 > a_2$, then $(\partial_x a_1)^{top} > (\partial_x a_2)^{top}$. This implies that the images $\partial_x a$ of the monomials $a \in \Theta^p_{d-1}$ are linearly independent in $\Theta_d^p$. Given a representative of a class in $\left( \frac{\Theta}{\partial_x \Theta}\right)^{p}_{d}$ we can express all the monomials of the form~\eqref{topbas} in terms of combinations of monomials of strictly lower lexicographic order. It follows that a basis can be chosen in the form~\eqref{quotbas}. 

By specialising to the case $p=3$, and spelling out the allowed sets of indexes, we obtain the statement of the lemma.

% OLD PROOF
%Let us consider first the case $d=2k+1$. A basis of $\Theta_d$ is given by monomials 
%\begin{equation}
%\theta^a \theta^b \theta^c , \quad a > b > c,
%\end{equation}
%with $a + b + c = d$. Consider the class in $\Theta \slash \partial_x \Theta$ associated with the representative $\theta^a \theta^b \theta^c$ with $b < a-1$. We have that 
%\begin{equation}
%\partial_x ( \theta^{a-1} \theta^b \theta^c ) = \theta^a \theta^b \theta^c + \theta^{a-1} \theta^{b+1} \theta^c + \theta^{a-1} \theta^b \theta^{c+1},
%\end{equation}
%therefore such class can be represented by a combination of monomials where the difference between the indexes of the first and second theta is strictly smaller. By repeating this process we conclude that the set of representatives given in the lemma is a set of generators. 
%
%We need to prove independence. Consider an element with representative 
%\begin{equation}
%\sum_{l=m}^{\lfloor\frac{k-2}3 \rfloor} b_l \theta^{k-l} \theta^{k-l-1} \theta^{2l} .
%\end{equation}
%Such element is zero in $\Theta \slash \partial_x \Theta$ iff its variational derivative w.r.t. $\theta$ is zero. Computing such variational derivative we get 
%\begin{equation}
%2 b_m (-1)^{k-m} \theta^{2k-2m-1} \theta^{2m} + \dots,
%\end{equation}
%where the dots contain only $\theta^s$ with $s<2k-2m-1$. We conclude that $b_m =0$, and, by induction, all constants $b_l$ must vanish. 
%
%The proof for the $d=2k$ case is completely analogous. 
\end{proof}

It follows that a basis of $\cB \left( \frac{\Theta}{\partial_x \Theta}\right)^{3}_{d}$ is given by the elements
\begin{equation}
\cB (\theta^{(a,0)} \theta^{(b,0)} \theta^{(c,0)} ) = u^{(a,0)} \theta^{(b,0)} \theta^{(c,0)} - u^{(b,0)}\theta^{(a,0)}\theta^{(c,0)} + u^{(c,0)}\theta^{(a,0)}\theta^{(b,0)},
\end{equation}
for indices $a, b, c$ chosen as in the basis above.

%\begin{remark}
%We give here a few comments on the proofs of the statements that we have recalled so far, in the $D=2$ case. 
%
%Notice that the only nontrivial fact here is that 
%\begin{equation}
%\ker \left\{ \partial_y : \hcF_1 \to \hcF_1 \right\} = \R ,
%\end{equation}
%that in this case can be proved directly. 
%
%\gc{maybe prove again that these are exact}
%
%\end{remark}

\section{Proof of the main Theorem}\label{sec:main}

Let us first reformulate our main statement in the $\theta$-formalism.

The Poisson bracket of Dubrovin-Novikov type of the form~\eqref{leading} corresponds to the element 
\begin{equation} \label{leadingtheta}
\fp_1 = \frac12 \iint \theta \theta^{(0,1)} \ dx dy
\end{equation}
in $\hcF^2_1$. 
The bivector $\delta^{(2k+1)}(x^1-y^1)  \delta(x^2-y^2)$ corresponds to the element in $\hcF^2_{2k+1}$ given by
\begin{equation}
\fp_{2k+1} = \frac12 \iint \theta \theta^{(2k+1,0)} \ dx dy.
\end{equation}

Therefore the normal form~\eqref{normalform} in $\theta$-formalism corresponds to the element 
\begin{equation} \label{normaltheta}
\fp(c) = \fp_1 + \sum_{k\geq1} c_k \fp_{2k+1}
\end{equation}
in $\hcF^2$.

The proof of Theorem~\ref{maintheorem} reduces to prove the three statements listed in Remark~\ref{steps}.

\subsection{}

Our first observation is:
\begin{lemma}
The bivectors $\fp_{2k+1}$ with $k\geq0$ are pairwise compatible, i.e.,
\begin{equation}
[ \fp_{2n+1}, \fp_{2m+1} ] = 0 , \qquad n, m \geq 0. 
\end{equation}
\end{lemma}
\begin{proof}
The Poisson bivectors $\fp_k$ do not depend on $u$ and its derivatives, therefore the variational derivatives w.r.t. $u$ appearing in the definition of Schouten-Nijenhuis bracket are vanishing.
\end{proof}
It clearly follows that $\fp(c)$ is a Poisson bivector for any choice of the constants $c=(c_1, c_2, \dots )$.

\subsection{}

Next we show that for any distinct choice of the constants $c=(c_1, c_2, \dots )$ the corresponding bivector $P$ belongs to a different equivalence class under Miura transformations. 
\begin{proposition}
Let $\fp(c)$, resp. $\fp(\tilde{c})$, be the Poisson bivector of the normal form~\eqref{normaltheta} corresponding to a choice $c=(c_1, c_2, \dots )$, resp. $\tilde{c}=(\tilde{c}_1, \tilde{c}_2, \dots )$, of constants. If the two sequences $c$ and $\tilde{c}$ are not identically equal, then there is no Miura transformation of the second kind which maps $\fp(c)$ to $\fp(\tilde{c})$.
\end{proposition}

\begin{proof}
Assume there is a Miura transformation of the second kind mapping $\fp(c)$ to $\fp(\tilde{c})$, i.e.,
\begin{equation}
e^{\ad_X} \fp(c) =\fp(\tilde{c}),
\end{equation}
for $X \in \hcF^1_{\geq1}$. This identity can be rewritten as
\begin{equation}
\left( \frac{e^{\ad_X} - 1}{\ad_X} \right) \ad_X \fp(c) = \fp(\tilde{c}) - \fp(c).
\end{equation}
The operator inside the brackets has the form
\begin{equation}
\left( \frac{e^{\ad_X} - 1}{\ad_X} \right) = 1 + \frac12 \ad_X + \dots 
\end{equation}
therefore we can invert it. We obtain
\begin{equation} \label{invert}
 \ad_X \fp(c) = \left( \frac{e^{\ad_X} - 1}{\ad_X} \right)^{-1} \left( \fp(\tilde{c}) - \fp(c) \right).
\end{equation}
By assumption the two sequences $c$ and $\tilde{c}$ are not identically equal, hence there exists a smallest index $k$ for which $c_k \not = \tilde{c}_k$. It follows that 
\begin{equation}
\fp(\tilde{c}) - \fp(c) = (\tilde{c}_{k} - c_k ) \fp_{2k+1} + \dots ,
\end{equation}
where the dots denote terms of standard degree greater than $2k+1$. We conclude that $\ad_X \fp(c)$ has to vanish in standard degree less or equal to $2k$, i.e.,
\begin{equation} \label{constraint}
( \ad_X \fp(c) )_{\leq 2k} =0.
\end{equation}
So, the leading order term in the standard degree in~\eqref{invert} is 
\begin{equation} \label{keypoint}
(\ad_X \fp(c))_{2k+1} = (\tilde{c}_k - c_k) \fp_{2k+1}.
\end{equation}
The key point of the proof is to prove that the lefthand side is a $\ad_{\fp_1}$ coboundary, which leads to a contradiction since we know that $\fp_{2k+1}$ is a nontrivial class in $H^2_{2k+1}(\hcF, \fp_1)$. 

Notice that the lefthand side in~\eqref{keypoint} can be written
\begin{equation} \label{keypointexp}
\ad_{\fp_1} X_{2k} + \sum_{s=1}^{k-1} c_s \ad_{\fp_{2s+1}} X_{2(k-s)} = 
(\tilde{c}_k -c_k) \fp_{2k+1}   
\end{equation}
hence it is sufficient to prove that the sum in the lefthand side is in the image of $\ad_{\fp_1}$. 

Equation~\eqref{constraint} gives a sequence of constraints on $X$. Let us consider in particular the constraints with odd degree 
\begin{equation}
(\ad_X \fp(c))_{2s+1}=0 ,\quad 
s=1, \dots , k-1,
\end{equation}
which can be written
\begin{equation} \label{explic}
\ad_{\fp_1} X_{2s} +\sum_{l=1}^{s-1} c_{l} \ad_{\fp_{2l+1}} X_{2(s-l)} =0.
\end{equation}
This equation for $s=1$ simply says that $X_2$ is a cocycle w.r.t. $\ad_{\fp_1}$,
\begin{equation}
\ad_{\fp_1} X_2 = 0.
\end{equation}
By the vanishing of the Poisson cohomology $H^1_2(\hcF, \fp_1$), $X_2$ is necessarily a coboundary, i.e.,
\begin{equation}
X_2 = \ad_{\fp_1} f_1
\end{equation} 
for some $f_1 \in \hcF^0_1$. 

More generally, we have that for each $s=1, \dots , k-1$
\begin{equation} \label{induc}
X_{2s} = \ad_{\fp_1} f_{2s-1}+\sum_{l=1}^{s-1} c_{l} \ad_{\fp_{2l+1}} f_{2(s-l)-1} 
\end{equation}
for some $f_{2l-1} \in \hcF^0_{2l-1}$, $l=1,\dots ,2s-1$. We can prove this by induction. Let us therefore assume that~\eqref{induc} holds for $s=1, \dots , t-1$ for $t  \leq k-1$, and show that it holds for $s=t$ too. Substituting the inductive assumption in~\eqref{explic} for $s=t$ we get %(after removing a vanishing term analogous to~\eqref{vanter}) 
that
\begin{equation}
\ad_{\fp_1} \left( 
X_{2t} - \sum_{l=1}^{t-1} c_l \ad_{\fp_{2l+1}} f_{2(t-l)-1} 
\right) = 0.
\end{equation}
The expression inside the brackets is therefore a cocycle, which has to be a coboundary due to the triviality of $H^1_{2t}(\hcF, \fp_1)$, i.e.,
\begin{equation}
 X_{2t} - \sum_{l=1}^{t-1} c_l \ad_{\fp_{2l+1}} f_{2(t-l)-1} = \ad_{\fp_1} f_{2t-1},
\end{equation}
for some $f_{2t-1} \in \hcF^0_{2t-1}$. This gives~\eqref{induc} for $s=t$.

Substituting~\eqref{induc} in~\eqref{keypointexp}, we get that 
\begin{equation} \label{conclu}
(\tilde{c}_k -c_k) \fp_{2k+1}   = \ad_{\fp_1} \left( 
X_{2k} - \sum_{s=1}^{k-1} c_s \ad_{\fp_{2s+1}} f_{2(k-s)-1}
\right),
\end{equation}
up to a term that can be written as
\begin{equation} \label{vanter}
\sum_{n\geq2}^{k-1} \left[ 
\sum_{\substack{   s,l \geq 1\\ s+l = n}} c_s c_l \ad_{\fp_{2s+1}} \ad_{\fp_{2l+1}}
\right] f_{2(k-n)-1}
\end{equation}
and therefore clearly vanishes. Equation~\eqref{conclu} leads to sought contradiction.

 The Lemma is proved.
\end{proof}

%sergey part

\subsection{}

Finally we prove that any Poisson bivector with leading order $\fp_1$ given by~\eqref{leadingtheta} can always be brought to the form~\eqref{normaltheta} by a Miura transformation of the second kind.

\begin{proposition} \label{step3}
Let $P \in \hcF^2_{\geq1}$ be a Poisson bivector with degree one term  equal to $\fp_1$. Then there is a Miura transformation that maps $P$ to a $\fp(c)$ for a choice of constants $c=(c_1, c_2, \dots )$.
\end{proposition}
\begin{proof}
The Poisson bivector $P\in\hcF^2_{\geq1}$ has to satisfy 
$[P,P]=0$. 
We want to show by induction that, taking into account this equation, it is possible, by repeated application of Miura transformations, to put all terms in normal form and to kill all terms that come from the Bockstein homomorphism. 

Let us denote by $\fp_{(s)}(c_1,\dots,c_{\lfloor s/2-1 \rfloor})$ a bivector of the form
\begin{equation}
\fp_{(2k)}(c_1, \dots ,c_{k-1}) = 
\fp_1 + \sum_{l=1}^{k-1} c_l \fp_{2l+1}
+ \sum_{l=k+1}^{2k-1} Q_l
+ P_{2k} + \dots  ,
\end{equation}
\begin{equation}
\fp_{(2k+1)}(c_1, \dots ,c_{k-1}) = 
\fp_1 + \sum_{l=1}^{k-1} c_l \fp_{2l+1}
+ \sum_{l=k+1}^{2k} Q_l
+ P_{2k+1} + \dots  ,
\end{equation}
for $s$ respectively even or odd, where $Q_l \in \cB \left( \frac{\Theta}{\partial_x \Theta} \right)^3_l$ , $P_l \in \hcF^2_l$, the dots denote higher order terms, and 
\begin{equation} \label{fpfp}
[\fp_{(s)}, \fp_{(s)} ]=0. 
\end{equation}
The inductive hypothesis is valid for $s=2$, indeed $\fp_{(2)}$ is exactly of the required form.

Let us now show that by a Miura transformation a Poisson bivector of the form $\fp_{(s)}$ can be made of the form $\fp_{(s+1)}$. 

When $s=2k$ is even, in degree $2k+1$ the equation~\eqref{fpfp} gives
\begin{equation}
[ \fp_1 , P_{2k} ] 
+ \sum_{\substack{2l+m=2k \\ 1\leq l \leq k-1\\ k+1 \leq m \leq 2k-1 }} 
[c_l \fp_{2l+1}, Q_{m} ] =0.
\end{equation}
The first observation is that both terms above need to be separately zero. This follows from the fact that the first term has nonzero degree  in the number of derivatives w.r.t. $y$, while the second term has degree zero. 

By Corollary~\ref{corcoh} the cohomology $H^2_{2k}(\hcF)$ is given only by elements coming from the Bockstein homomorphism, therefore exists $Q_{2k} \in \cB \left( \frac{\Theta}{\partial_x \Theta} \right)^3_{2k}$ such that 
$P_{2k} +\ad_{\fp_1} X_{2k-1} = Q_{2k}$ 
for some $X_{2k-1} \in \hcF^1_{2k-1}$.

Acting with the Miura transformation $e^{\ad_{X_{2k-1}}}$ on $\fp_{(2k)}$ we get a new Poisson bivector, where the terms of degree less or equal to $2k-1$ are unchanged, the term $P_{2k}$ has been replaced with the term $Q_{2k}$, and the terms of higher order are in general different. 
We have therefore that $\fp_{(2k+1)} = e^{\ad_{X_{2k-1}}} \fp_{(2k)}$ is of the form above, as required. 

When $s=2k+1$ is odd, in degree $2k+2$ from~\eqref{fpfp} we get
\begin{equation} \label{oddvanish}
[ \fp_1, P_{2k+1} ] + \sum_{\substack{2l+m=2k+1 \\ 1 \leq l \leq k-1 \\ k+1 \leq m \leq 2k } } 
[c_l \fp_{2l+1} , Q_m] 
+\frac12 [Q_{k+1},Q_{k+1}] =0 .
\end{equation}

As in the previous case, the first term has to vanish, hence $P_{2k+1}$ is an $\ad_{\fp_1}$-cocycle. The cohomology $H^2_{2k+1}(\hcF)$ decomposes in two parts, therefore there is a constant $c_k$ and an element  $Q_{2k+1}$ in $\cB \left( \frac{\Theta}{\partial_x \Theta} \right)^3_{2k+1}$ such that 
$P_{2k+1} + \ad_{\fp_1} X_{2k} = c_k \fp_{2k+1} + Q_{2k+1}$ 
for some $X_{2k} \in \hcF^1_{2k}$.

The second and third term in~\eqref{oddvanish} have also to be both zero. This follows from the fact that they have different degree in the number of $u^{(s,t)}$. As we have seen in Section~\ref{sec:poisson}, the elements $Q_k$ are linear in the variables $u^{(s,t)}$, while the elements $\fp_{k}$ do not contain them. 

From the vanishing of the last term, $[Q_{k+1}, Q_{k+1}]= 0$, we finally derive that $Q_{k+1}$ is zero. This is guaranteed by Lemma~\ref{nontriv}. The proof of this Lemma, being quite technical, is given in Section~\ref{techsec}.

Taking into account this vanishing, the action of the Miura transformation $e^{\ad_{X_{2k}}}$ on $\fp_{(2k+1)}$ gives exactly the term $\fp_{(2k+2)}$.

By induction we see that we can continue this procedure indefinitely, therefore we conclude that we cannot have any non-trivial deformation coming from $\left( \frac{\Theta}{\partial_x \Theta} \right)^3$ via the Bockstein homomorphism, and that the Miura transformation $\cdots e^{\ad_{X_2}} e^{\ad_{X_1}}$ given by the composition of the Mira transformations defined above, sends the original Poisson bivector $P = \fp_1 + \dots$ to a Poisson bivector of the form $\fp(c)$ for a choice of constants $c_1, c_2, \dots$. 

The Proposition is proved. 
\end{proof}

\subsection{}
\label{techsec}

In this section we prove the following statement, which is essential in the proof of Proposition~\ref{step3}:

\begin{lemma}
\label{nontriv}
Let $\chi \in \left( \frac{\Theta}{\partial_x \Theta} \right)^3_d$ and $\cB(\chi)$ its image through the map~\eqref{BBdef} in $\hcF^2_d$. 
If $[\cB(\chi),\cB(\chi) ]=0$, then $\chi =0$.
\end{lemma}
\begin{proof}
We have
\begin{align}
[ \cB(\chi) ,\cB(\chi) ] &= 2 \iint \frac{\delta \cB(\chi)}{\delta \theta} \frac{\delta \cB(\chi)}{\delta u} 
= 2 \iint  \frac{\delta \cB(\chi)}{\delta \theta} \frac{\delta \chi}{\delta \theta} 
\\
&= - 2 \iint \cB\left( \frac{\delta \chi}{\delta \theta} \right) \frac{\delta \chi}{\delta \theta} 
= - \iint \cB \left(\frac{\delta \chi}{\delta \theta}  \right)^2,
\label{BBzzz}
\end{align}
where the second and third equalities follow from the simple identities
\begin{equation}
\frac{\delta \cB(\chi)}{\delta u} = \frac{\delta \chi}{\delta \theta} ,
\qquad
[\cB , \frac{\delta }{\delta \theta}]_+ = 0.
\end{equation}
Since we proved that the map 
\begin{equation}
\cB : \left( \frac{\Theta}{\partial_x \Theta}\right)^{p+1}_d \to H^p_d(\hcF) 
\end{equation}
is injective, the vanishing of~\eqref{BBzzz} implies that $\left(\frac{\delta \chi}{\delta \theta}  \right)^2=0$ in $ \left( \frac{\Theta}{\partial_x \Theta}\right)^{4}$. From this fact it follows that $\chi =0$, as we prove in the remaining part of this section\footnote{Notice that this fact, in the case of standard differential polynomials in commuting variables, follows from a simple observation: the derivative in $x$ of a differential polynomial cannot be a square, since it has to be linear in the highest derivative. In the case of anticommuting variables however, a quite involved proof is necessary.}. 

%Consider the space $\Theta_d^p$ spanned by the monomials in $\theta^i:= \partial_x^i \theta$, $i=0,1,2,\dots$. Note the we restrict ourselves for the one independent variable for the moment. 

Let $\mathrm{sq}\colon \Theta^2_k\to \Theta^4_{2k}$ be the map that sends an element $\alpha\in\Theta^2_k$ to $\alpha^2\in \Theta^4_{2k}$. In the rest of this section we will use the notation $\theta^d = \theta^{(d,0)}$.

\begin{lemma} \label{lem:square}
	The intersection of $\mathrm{sq}(\Theta^2_k)$ and $\partial_x\Theta^4_{2k-1}$ is equal to zero. In other words, if $\alpha\in \Theta^2_k$ and $\alpha^2$ is $\partial_x$-exact, then $\alpha^2=0$ and, therefore, $\alpha$ is proportional to a monomial $\theta^i\theta^{k-i}$ for some $i=1,\dots,\lfloor \frac{k-1}2 \rfloor$. 
\end{lemma}
\begin{proof}
A basis in $\Theta_{2k-1}^4$ is given by standard monomials $\theta^{i_1} \theta^{i_2} \theta^{i_3} \theta^{i_4}$ with total degree $i_1 + i_2+ i_3+i_4 = 2k-1$. By {\it standard} monomial we indicate a monomial where the indices are ordered as $i_1 > i_2 > i_3 > i_4 \geq0$ to avoid duplicates. 
 
We can write $\Theta_{2k-1}^4 = \cV_1 \oplus \cV_2$, where a basis for $\cV_1$ is given by standard monomials with the restriction $i_1+i_4 \leq k-1$, and a basis for $\cV_2$ is given by standard monomials with  $i_1 +i_4 \geq k$. 

It is convenient to define also the subspace $\cW$ of $\Theta_{2k}^4$ which is spanned by all monomials that appear in the $\partial_x \cV_1$; more explicitly $\cW$ is generated by the monomials 
\begin{equation}
\theta^{i_1+1} \theta^{i_2} \theta^{i_3} \theta^{i_4}, \quad 
\theta^{i_1} \theta^{i_2+1} \theta^{i_3} \theta^{i_4}, \quad 
\theta^{i_1} \theta^{i_2} \theta^{i_3+1} \theta^{i_4}, \quad 
\theta^{i_1} \theta^{i_2} \theta^{i_3} \theta^{i_4+1},
\end{equation}
with $i_1 > i_2 > i_3 > i_4 \geq0$, $i_1+i_4 \leq k-1$, and $i_1 + i_2+ i_3+i_4 = 2k-1$.

We denote by $\Theta_k^2 \cdot \Theta_k^2$ the subspace of $\Theta_{2k}^4$ spanned by standard monomials $\theta^{i_1} \theta^{i_2} \theta^{i_3} \theta^{i_4}$ with $i_1 > i_2 > i_3 > i_4 \geq0$ and $i_1 + i_2+ i_3+i_4 = 2k$ with $i_1+i_4 = k$ and $i_2 + i_3 = k$. 
It is indeed the subspace given by the product of two arbitrary elements of $\Theta_k^2$.

Clearly, both $\partial_x \cV_1$ and $\Theta_k^2 \cdot \Theta_k^2$ are subspaces of $\cW$.

Let us now prove that $\partial_x \cV_2$ has zero intersection with $\cW$. 
Let $v = \sum_\gamma v_\gamma \,  \gamma$ be an element in $\cV_2$, where $\gamma$ is in the standard basis of $\cV_2$ described above. 
Let $\partial_x v = \sum_\gamma v_\gamma \, \partial_x \gamma \in \cW$.  We have already seen that the elements $\partial_x \gamma$ are linearly independent. If $\gamma = \theta^{i_1} \theta^{i_2} \theta^{i_3} \theta^{i_4}$ then $\partial_x \gamma$ is equal to $\theta^{i_1+1} \theta^{i_2} \theta^{i_3} \theta^{i_4}$ plus  lexicographically lower terms. The lexicographically leading order term is therefore of a standard monomial $\theta^{j_1} \theta^{j_2} \theta^{j_3} \theta^{j_4}$ with $j_1+j_4 \geq k+1$. 
But all basis elements in $\cW$ are standard monomials with $j_1+j_4 \leq k$.
It follows that, if $\gamma$ is the lexicographically highest term in $v$, we must have $v_\gamma=0$. By induction $v$ vanishes. 

The two facts $\partial_x \cV_1 \subseteq \cW$ and $\partial_x \cV_2 \cap \cW = (0)$ imply at once that the preimage $\partial_x^{-1}( \cW )$ in $\Theta_{2k-1}^4$ is contained in $\cV_1$, and the same holds for $\Theta_k^2 \cdot \Theta_k^2$ since it is a subspace of $\cW$, i.e., we have
\begin{equation}
\partial_x^{-1} ( \Theta_k^2 \cdot \Theta_k^2 ) \subseteq \cV_1. 
\end{equation}
Since $\mathrm{sq}(\Theta_k^2) \subseteq \Theta_k^2 \cdot \Theta_k^2$, our original problem reduces to finding the intersection of $\mathrm{sq}(\Theta_k^2)$ and $\partial_x \cV_1$. 

Let $\alpha = \sum_{i = \lceil \frac{k+1}2 \rceil}^{k} \alpha_i\, \theta^i \theta^{k-i}$ be an element of $\Theta_k^2$ whose square is in $\partial_x \cV_1$. We want to show that at most one of the coefficients $\alpha_i$ is non zero. We therefore assume that at least two such coefficients are non zero and show that it leads to a contradiction. Let $s$ be the higher index for which $\alpha_s \not= 0$ and $t <s$ the second higher index for which $\alpha_t \not=0$. 

Denote by $\cW^{(j)}$ the subspace of $ \Theta_k^2 \cdot \Theta_k^2$ spanned by monomials of the form
\begin{equation}
\theta^{i} \theta^{j} \theta^{k-j} \theta^{k-i} \text{ for }  i= k, \dots , j+1 ,
\end{equation}
and denote by $\widetilde{\cW}$ the space spanned by the basis monomials in $\cW$ which are not in $ \Theta_k^2 \cdot \Theta_k^2$.
Notice that
\begin{equation}
 \Theta_k^2 \cdot \Theta_k^2 = \bigoplus_{j=\lceil \frac{k+1}2 \rceil}^{k-1} \cW^{(j)} ,
\end{equation}
and consequently
\begin{equation}
\cW = \widetilde{\cW} \oplus \bigoplus_{j=\lceil \frac{k+1}2 \rceil}^{k-1} \cW^{(j)}.
\end{equation}

Observe that a monomial $\theta^{i} \theta^{j} \theta^{k-j} \theta^{k-i}$ in $\cW^{(j)}$ can appear in the $\partial_x$-image of four different monomials in $\Theta^4_{2k-1}$ but only two of them are elements of $\cV_1$, i.e., 
\begin{equation} \label{twomon}
 \theta^{i-1} \theta^{j} \theta^{k-j} \theta^{k-i} , \quad 
 \theta^{i} \theta^{j} \theta^{k-j} \theta^{k-i-1},
\end{equation}
so we only need to consider these two. 

Notice that a monomial in $\cV_1$ of such form, i.e., $ \theta^{l} \theta^{j} \theta^{k-j} \theta^{k-l-1}$, is mapped by $\partial_x$ to the sum of four monomials, two of which are in $\cW^{(j)}$, i.e.,
\begin{equation}
\theta^{l+1} \theta^{j} \theta^{k-j} \theta^{k-l-1}, \quad 
\theta^{l} \theta^{j} \theta^{k-j} \theta^{k-l},
\end{equation}
and two are in $\widetilde{\cW}$.

%
%
%
%Let us define $\cV_1^{(j)}$ to be the subspace of $\cV_1$ with basis given by basis monomials of $\cV_1$ with the extra condition that $i_2 = j$ and $i_3 = k-j$. Clearly, the two monomials~\eqref{twomon} are in  $\cV_1^{(j)}$.
%%
%Moreover the $\partial_x$-image of each monomial $\theta^{i_1} \theta^{j} \theta^{k-j} \theta^{i_4}$ in  $\cV_1^{(j)}$ is given by the sum of four monomials, two of which are in $\cW^{(j)}$, i.e.,
%\begin{equation}
%\theta^{i_1+1} \theta^{j} \theta^{k-j} \theta^{i_4}, \quad 
%\theta^{i_1} \theta^{j} \theta^{k-j} \theta^{i_4+1},
%\end{equation}
%and two are in $\widetilde{\cW}$.

Since $\alpha^2 \in  \Theta_k^2 \cdot \Theta_k^2$, it can be decomposed in its components $(\alpha^2)_j \in \cW^{(j)}$, and we have in particular that
\begin{equation}
(\alpha^2)_t = 2 \alpha_s \alpha_t \,  \theta^s \theta^t \theta^{k-t} \theta^{k-s} ,
\end{equation}
since we have assumed that $\alpha_i=0$ for $i>s$ and $t<i<s$.

All these observations imply that there must be an element $\beta$ of $\cV_1$ of the form
\begin{equation}
\beta = \sum_{i=k-1}^{t+1} \beta_i \theta^{i} \theta^{t} \theta^{k-t} \theta^{k-i-1}
\end{equation}
such that its image through $\partial_x$ gives $(\alpha^2)_t$ plus some element in $\widetilde{\cW}$.

The lexicographically higher term in $\beta$, i.e., for $i=k-1$, is sent by $\partial_x$ to a term proportional to $\theta^{k} \theta^{t} \theta^{k-t} \theta^{0}$, which does not appear in $(\alpha^2)_t$, therefore $\beta_{k-1}=0$. Proceeding like this we set to zero all the constants $\beta_{k-1}, \dots, \beta_{s}$. 
Similarly, we can proceed from the lower part of the chain and set to zero all the remaining constants $\beta_{t+1}, \dots , \beta_{s-1}$.
But then $\beta=0$, therefore $\alpha_s \alpha_t =0$ and we are led to a contradiction.

We have proved that at most one of the constants $\alpha_i$ can be non zero. In such case $\alpha^2 =0$. The Lemma is proved.	
\end{proof}

\begin{lemma} \label{lem:varder}
	Consider an arbitrary element $\chi\in\Theta^3_d$. If $\frac{\delta \chi}{\delta \theta} = c\cdot \theta^i\theta^{d-i}$ for some $i=0,1,\dots,\lfloor \frac{d-1}{2}\rfloor$, then $c=0$.
\end{lemma}

\begin{proof} 

Consider the basis of $\left( \frac{\Theta}{\partial_x \Theta} \right)^3_d$ given in Lemma~\ref{lembas}, and the basis 
\begin{equation}
 \theta^d\theta^0,\theta^{d-1}\theta^1,\theta^{d-2}\theta^2,\dots
\end{equation}
of $\Theta^2_d$.
For this choice of bases the map $\frac{\delta}{\delta\theta}$ has a two-step triangular structure. In order to explain that, let us consider the two cases of odd and even $d$ separately. 

Consider first the $d=2k+1$ case. 
One can check\footnote{Note that the computation is slightly different in the case $3l=k-1$.} that the variational derivative $\frac{\delta}{\delta \theta}$ of a basis element $ \theta^{k-l+1} \theta^{k-l} \theta^{2l} $, with $3l < k$, is equal to 
\begin{equation}
2 (-1)^{k-l+1} \theta^{d-2l} \theta^{2l} + (d-2l) (-1)^{k-l+1} \theta^{d-2l-1} \theta^{2l+1} 
\end{equation}
plus terms which are of lower lexicographic order.  Notice that the coefficients of the two monomials above are non-vanishing. 

Observe that $\frac{\delta}{\delta\theta} \theta^{k+1}\theta^k\theta^0$ contains the monomials $\theta^d\theta^0$ and $\theta^{d-1}\theta^1$, while the variational derivatives of all other basis elements with $l\geq 1$ can not contain $\theta^d\theta^0$ and $\theta^{d-1}\theta^1$. 
Thus, if $\frac{\delta \chi}{\delta \theta} = c\cdot \theta^i\theta^{d-i}$ for some $i$, then the coefficient of $\theta^{k+1}\theta^k\theta^0$ in $\chi$ has to be equal to zero. 

We can continue this process by induction. Assume that we have already proved that the first $l$ elements of the basis cannot appear in $\chi$. Then the variational derivative of the basis element $ \theta^{k-l+1} \theta^{k-l} \theta^{2l} $ is the only one that contains $\theta^{d-2l} \theta^{2l} $ and $ \theta^{d-2l-1} \theta^{2l+1} $. It follows from the same reason as above, that such basis element cannot appear in $\chi$.

	In the case $d=2k$ we can apply the same reasoning. In this case  the variational derivative $\frac{\delta}{\delta \theta}$ of a basis element $\theta^{k-l} \theta^{k-l-1} \theta^{2l+1}$, with $3l < k-2$, is equal to
\begin{equation}
 2 (-1)^{k-l} \theta^{d-2l-1} \theta^{2l+1} + (d-2l-1) (-1)^{k-l} \theta^{d-2l-2} \theta^{2l+2} 
\end{equation}
plus terms of lower lexicographic order. Notice that $\theta^d\theta^0$ never enters the image of any basis element in $\Theta^3_{d}/\partial_x\Theta^{3}_{d-1}$. Since the coefficients of the two monomials above are non-vanishing, we can
%Then, $\theta^{2k-1-2j}\theta^{2j+1}$ and $\theta^{2k-2-2j}\theta^{2j+2}$, $j\geq 0$, enter with nontrivial coefficients only the the images of the monomials $\theta^{k-i}\theta^{k-i-1}\theta^{2i+1}$ for $i\geq j$. In the case $i=j$, the coefficients are equal to $(-1)^{k-j}2$ and $(-1)^{k-j}(2k-2j-1)$, respectively, that is, both coefficients are non-trivial. This allows us to
 apply the same argument as in the case of odd $d$, mutatis mutandis. 
\end{proof}
%
%\begin{corollary}\label{cor:triviality-dx-square}
%	Consider an arbitrary element $\xi\in\Theta^3_d$. If $(\frac{\delta \xi}{\delta \theta} )^2$ belongs to the image of $\partial_x$, then $\xi$ itself belongs to the image of $\partial_x$. 
%\end{corollary}

Now let us consider an arbitrary element $\chi\in\Theta^3_d$, such that $(\frac{\delta \chi}{\delta \theta} )^2$ belongs to the image of $\partial_x$.
From Lemma~\ref{lem:square} it follows that $\frac{\delta \chi}{\delta \theta} =c\cdot \theta^i\theta^{d-i}$ for some $i=0,1,\dots,\lfloor d/2\rfloor$. Then Lemma~\ref{lem:varder} implies that $\frac{\delta \chi}{\delta \theta} =0$, hence $\chi$ belongs to the image of $\partial_x$.

We have proved that $\chi=0$ as element of $\left( \frac{\Theta}{\partial_x \Theta} \right)^3_d$. Lemma~\ref{nontriv} is proved. 
\end{proof}

\section{The numerical invariants of the Poisson bracket}

In principle all the numerical invariants of a Poisson bracket of the form~\eqref{def}, namely the sequence $(c_1,c_2,\ldots)$, can be extracted iteratively solving order by order for the Miura transformation which eliminates the coboundary terms. 
Providing a general formula for the invariants of a Poisson bivector is hard, since the elimination of each coboundary term affects in principle all the higher order ones and it is necessary to give an explicit form for the Miura transformation. However, the lowest invariants can be computed as follows.

\begin{proposition}\label{numinv}
Consider a Poisson bracket of the form
\begin{equation}\label{Pexp}
\begin{split}
\{ u(x^1,x^2) &, u(y^1,y^2) \}=\{ u(x^1,x^2) , u(y^1,y^2) \}^0 +\\
&+\sum_{k>0} \epsilon^k \sum_{\substack{k_1,k_2 \geq0 \\ k_1 +k_2 \leq k+1}}  A_{k; k_1, k_2}(u(x)) \delta^{(k_1)}(x^1-y^1)  \delta^{(k_2)}(x^2-y^2),
\end{split}
\end{equation}
as in \eqref{def}. Here $A_{k; k_1,k_2} \in \cA$ and $\deg A_{k; k_1,k_2} = k - k_1 - k_2 +1$.
Then the first numerical invariants of the bracket, giving the normal form of Theorem~\ref{maintheorem}, are
\begin{align}\label{c1-thm}
c_1&=A_{2;3,0}, \\\label{c2-thm}
c_2&=A_{4;5,0}(u)-A_{2;3,0}A_{2;2,1}(u).
\end{align}
%\mc{I would give up the idea of explicitly finding $c_3$ -- it does not seem possible to avoid solving for $X_3$ and $X_4$.}
\end{proposition}
Notice that $A_{2;3,0}$ is implied to be a constant.
\begin{proof}
We recall that, given a Poisson bracket $P$ of form \eqref{def}, it can be expanded according to its differential order. For notational compactness, we will denote
\begin{equation}\label{Pexp-pf}
P_{k+1}:=\sum_{\substack{k_1,k_2 \geq0 \\ k_1 +k_2 \leq k+1}}  A_{k; k_1, k_2}(u(x)) \delta^{(k_1)}(x^1-y^1)  \delta^{(k_2)}(x^2-y^2)
\end{equation}
for $k>0$, so that $\deg P_k=k$.

In this proof, we replace $(x^1,x^2)$ with $(x,y)$ as we did in the previous sections; moreover, with a slight abuse of notation we identify the Dirac's delta derivatives with the corresponding elements of $\hcF$ previously used
\begin{align*}
\fp_1&:=\delta(x^1-y^1)\delta^{(1)}(x^2-y^2)&\fp_k&:=\delta^{(k)}(x^1-y^1)\delta(x^2-y^2).
\end{align*}

Using this notation, the Schouten identity $[P,P]=0$ reads
\begin{equation}\label{eqSch-pf}
2[\fp_1,P_k]+\sum_{l=2}^{k-1}[P_l,P_{k-l+1}]=0
\end{equation}
for $k\geq2$. The first equation is $[\fp_1,P_2]=0$; we solved it in \cite{ccs15}, finding for $P_2$
\begin{align*}
A_{1;2,0}&=0&A_{1;1,1}&=0&A_{1;0,2}&=0\\
A_{1;1,0}&=-f(u)\partial_{y}u&A_{1;0,1}&=f(u)\partial_{x}u&A_{1;0,0}&=0
\end{align*}
for any function $f(u)$. Since $H^2_2(\hcF)=0$, we have $P_2=[X_1,\fp_1]$ and the Miura transformation that eliminates $P_2$ from $P$ is $e^{-\ad_{\epsilon X_1}}$. The evolutionary vector field $X_1$ has characteristic
\begin{equation}
X_1(u)=  F(u)\partial_{x}u
\end{equation}
where $F(u)=\int^u f(s) ds$. We also observe that $\ad_{X_1}^m\fp_1=0$ for $m>1$.

We apply the Miura transformation generated by $-\epsilon X_1$ to $P$ and get
\begin{equation}
\begin{split}
\tilde{P}=e^{-\ad_{\epsilon X_1}}P&=\fp_1+\epsilon^2 P_3+\epsilon^3\left(P_4-[X_1,P_3]\right)+\\
&+\epsilon^4\left(P_5-[X_1,P_4]+\frac{1}{2}[X_1,[X_1,P_3]]\right)+\cdots
\end{split}
\end{equation}
The first equation of the system \eqref{eqSch-pf} for $\tilde{P}$, and the results used in the proof of Lemma~\ref{step3} give us $P_3=c_1\fp_3+[X_2,\fp_1]$.

$[X_2,\fp_1]$ is a bivector whose degree in the number of derivatives w.r.t. $x^2$ is at least 1; notice that $x^1$ corresponds to $x$ and $x^2$ corresponds to $y$, in the notation of Section~\ref{sec:poisson} and \ref{sec:main}. Hence, we can write
\begin{equation}\label{P3-pf}
\begin{split}
P_3&=A_{2;3,0}(u)\fp_3+A_{2;2,1}(u)\delta^{(2)}(x^1-y^1)\delta^{(1)}(x^2-y^2)\\&\quad+A_{2;1,2}(u)\delta^{(1)}(x^1-y^1)\delta^{(2)}(x^2-y^2)+A_{2;0,3}(u)\delta(x^1-y^1)\delta^{(3)}(x^2-y^2)\\&\quad+\cdots\\
&=c_1\fp_3+[X_2,\fp_1]
\end{split}
\end{equation}
This equation immediately gives $A_{2;3,0}(u) = A_{2;3,0} = c_1$ as in \eqref{c1-thm}. Moreover, we can solve it for $X_2$; the characteristic of the evolutionary vector field is a differential polynomial with top degree w.r.t. the $x$ derivatives is $1/2\,A_{2;2,1}(u)\partial_{x}^2u+\tilde{A}(u)\left(\partial_{x}u\right)^2$. Here we are interested only in first summand because it is the one that gives the highest number of $x$-derivatives in $[X_2,\fp_r]$, for any $r$.

We apply to $\tilde{P}$ the Miura transformation $e^{-\ad_{\epsilon^2 X_2}}$ to eliminate the coboundary term of $P_3$ and are left with
\begin{equation}\label{Ptildetilde-pf}
\begin{split}
e^{-\ad_{\epsilon^2X_2}}\tilde{P}&=\fp_1+\epsilon^2c_1\fp_3+\epsilon^3\left(P_4-c_1[X_1,\fp_3]-[X_1,[X_2,\fp_1]]\right)+\\
&+\epsilon^4\left(P_5-[X_1,P_4]+\frac{1}{2}c_1[X_1,[X_1,\fp_3]]+\frac{1}{2}[X_1,[X_1,[X_2,\fp_1]]]\right.\\
&\left.-c_1[X_2,\fp_3]-\frac{1}{2}[X_2,[X_2,\fp_1]]\right)+\cdots
\end{split}
\end{equation}
We now use the fact that $H^2_4(\hcF)=0$ to get
$$
P_4=c_1[X_1,\fp_3]+[X_1,[X_2,\fp_1]]+[X_3,\fp_1]
$$
for some homogeneous vector field $X_3$ of degree $3$. This allows us to replace $P_4$ in \eqref{Ptildetilde-pf} and to apply the Miura transform $e^{-\ad_{\epsilon^3 X_3}}$ to it to get rid of the term $\epsilon^3$ in the expansion. The terms of order $<3$ are left unaffected by this transformation, while the coefficient of $\epsilon^4$ becomes
\begin{multline}
P_5-[X_1,[X_3,\fp_1]]-\frac{1}{2}c_1[X_1,[X_1,\fp_3]]-\frac{1}{2}[X_1,[X_1,[X_2,\fp_1]]]-\\
-c_1[X_2,\fp_3]-\frac{1}{2}[X_2,[X_2,\fp_1]]=c_2\fp_5+[X_4,\fp_1]
\end{multline}
where the equality is given by our results about $H^2_5(\hcF)$ and the proof of Lemma~\ref{step3}. The invariant $c_2$ must be read taking the coefficient of $\fp_5$ in the left hand side of the equation: this coefficient cannot be obtained by summands that are of $y$-degree bigger or equal to 1. Thus we focus on the summands
$$
P_5-\frac{1}{2}[X_1,[X_1,\fp_3]]-c_1[X_2,\fp_3]=c_2\fp_5+\cdots.
$$
A direct computation shows that in $\ad_{X_1}^2\fp_3$ the term $\fp_5$ does not appear, while it does appear in $[X_2,\fp_3]$. Using the form of $X_2$ we have previously derived, we find
$$
P_5=(A_{4;5,0}(u)\fp_5+\cdots)=\left(c_2+c_1A_{2;2,1}(u)\right)\fp_5+\cdots
$$
from which we get \eqref{c2-thm}.
\end{proof}

\begin{example}\label{eg}
We can compute all the numerical invariants when the Poisson bracket is particularly simple. Let us consider the bracket 
\begin{multline}\label{pb-ex}
\{u(x),u(y)\}=\delta(x^1-y^1)\delta'(x^2-y^2)+\delta'''(x^1-y^1)\delta(x^2-y^2)\\+\delta''(x^1-y^1)\delta'(x^2-y^2).
\end{multline}
Proposition~\ref{numinv} immediately tells us that $c_1=1$ and $c_2=-1$. Let us denote for brevity $\fp_{s,t}$ the bivector corresponding to $\frac{1}{2}\int\theta\theta^{(s,t)}$. The bivector corresponding to the bracket then reads $P=\fp_1+\fp_3+\fp_{2,1}$, and $\fp_{2,1}=\ad_{X_2}\fp_1$. It is very easy to derive $X_2=\frac{1}{2}u_{2x}\theta$. We have $\ad_{X_2}\fp_{s,t}=\fp_{s+2,t}$. The Miura transformation  $e^{-\ad_{X_2}}$ applied to $P$ gives
\begin{align}\notag
P_{(1)}&=\fp_1+\sum_{n=0}^\infty\frac{(-1)^n}{n!}\ad_{X_2}^n\fp_3+\sum_{n=1}^\infty(-1)^n\left(\frac{1}{n!}-\frac{1}{(n+1)!}\right)\ad_{X_2}^{n+1}\fp_1\\
&=\fp_1+\sum_{n=0}^\infty\frac{(-1)^n}{n!}\fp_{3+2n}+\sum_{n=1}^\infty(-1)^n\left(\frac{1}{n!}-\frac{1}{(n+1)!}\right)\fp_{2n+2,1}
\end{align}
Notice that the term $n=0$ in the first sum gives the only contribution of order 3, giving $c_1=1$. The further $\fp_1$-coboundary term should be read in the $n=1$ term of the second sum, namely for $-\frac{1}{2}\fp_{4,1}=\ad_{X_4}\fp_1$.
The next Miura transformation leads to
\begin{align*}
P_{(2)}&=\fp_1+\sum_{m=0}^\infty\sum_{n=0}^\infty\frac{(-1)^{n+2m}}{2^m m!n!}\fp_{2n+4m+3}+\sum_{m=1}^\infty\frac{(-1)^{2m}}{2^mm!}\fp_{4m,1}\\
&\quad+\sum_{m=0}^\infty\sum_{n=2}^\infty\left(\frac{1}{n!}-\frac{1}{(n-1)!}\right)\frac{(-1)^{n+2m}}{2^mm!n!}\fp_{2n+4m,1}.
\end{align*}
The procedure goes on -- always requiring us to find the vector field cancelling the lowest order term of the form $\fp_{s,1}$. At each step, we will need vector fields $X_{2s+2}$ such that
$$\ad_{X_{2s}}\fp_1=\frac{(-1)^{s+1}}{s}\fp_{2s,1}$$
and we obtain
$$P_{(\infty)}=\left(\prod_{s=1\ldots\infty}^{\curvearrowleft} e^{-\ad_{X_{2s}}}\right)\,P.$$
The Miura transformation cancels all the terms of the form $\fp_{s,1}$ and we are left with the following expression for the Poisson bivector brought to the normal form:
\begin{equation}
P_{(\infty)}=\fp_1+\sum_{m_1,m_2,\dots=0}^{\infty}\frac{(-1)^{m_1+2m_2+3m_3+\cdots}}{m_1!m_2!m_3!\cdots 2^{m_2}3^{m_3}\cdots}\fp_{3+2m_1+4m_2+6m_3+\cdots}
\end{equation}
We recall that $\frac{1}{2}\int\theta\partial_x^k\theta=\fp_k$. Hence, the infinite sum can be seen as a series expansion for $\frac{1}{2}\int\theta\partial_x^3/(1+\partial_x^2)\,\theta$ as follows:
\begin{multline*}
\frac{1}{2}\int\theta\partial_x^3\left(\sum_{m_1=0}^\infty\frac{(-1)^{m_1}}{m_1!}\partial_x^{2m_1}\right)\left(\sum_{m_2=0}^\infty\frac{(-1)^{2m_2}}{2^{m_2}m_2!}\partial_x^{4m_2}\right)\left(\sum_{m_3=0}^\infty\frac{(-1)^{3m_3}}{3^{m_2}m_3!}\partial_x^{6m_3}\right)\cdots\theta\\=
\frac{1}{2}\int\theta\left(\partial_x^3\, e^{-\partial_x^2+\frac{\partial_x^4}{2}-\frac{\partial_x^6}{3}+\cdots}\right)\theta=\frac{1}{2}\int\theta\,\partial_x^3\, e^{-\log(1+\partial_x^2)}\theta=\frac{1}{2}\int\theta\,\frac{\partial_x^3}{1+\partial_x^2}\theta.
\end{multline*}
We stress the fact that all these identities should always been understood in terms of formal power expansion. On the other hand, a more obvious expansion for the same expression is
$$
\frac{1}{2}\int\theta\,\frac{\partial_x^3}{1+\partial_x^2}\theta=\frac{1}{2}\int\theta\partial_x^3\sum_{k=0}^\infty(-1)^k\partial_x^{2k}\theta ,
$$
that translates into
\begin{equation}
P_{(\infty)}=\fp_1-\sum_{k=1}^\infty(-1)^k\fp_{2k+1} ,
\end{equation}
and gives us all the numerical invariants of \eqref{pb-ex}.
\end{example}

\appendix

\end{document}